\documentclass{article}
\usepackage{amsfonts}
\usepackage{amsmath}
\usepackage{graphicx}

\newtheorem{theorem}{Theorem}

\newtheorem{definition}[theorem]{Definition}
\newtheorem{example}[theorem]{Example}

\newenvironment{proof}[1][Proof]{\noindent \textbf{#1.} }{\  \rule{0.5em}{0.5em}}
\begin{document}

\begin{center}
\textbf{General form of the solutions of some difference equations via Lie
symmetry analysis}\\[0pt]
\vspace{1cm} Mensah Folly-Gbetoula, Melih G\"{o}cen, Mira\c{c} G\"{u}neysu%
\\[0pt]
\vspace{1cm} $^{1}$School of Mathematics, University of the Witwatersrand,
Wits 2050, Johannesburg, South Africa,

Mensah.Folly-Gbetoula@wits.ac.za\\
$^{2,3}$Department of Mathematics, Faculty of Arts and Sciences, Bulent
Ecevit University, Zonguldak, Turkey,

gocenm@hotmail.com, gnsu\_mrc@hotmail.com \vspace{1cm}

\textbf{Abstract}
\end{center}

In this paper, we obtain exact solutions of the following rational
difference equation

\begin{equation*}
x_{n+1}=\frac{x_{n}x_{n-2}x_{n-4}}{%
x_{n-1}x_{n-3}(a_{n}+b_{n}x_{n}x_{n-2}x_{n-4})},  \label{denklem}
\end{equation*}%
where $a_{n}$ and $b_{n}$ are random real sequences, by using the technique of Lie symmetry analysis. Moreover, we discuss the periodic nature and behavior of solutions  for some special cases. This work is
a generalization of some works by Elsayed and Ibrahim in [E.M.Elsayed, T. F.
Ibrahim, { Solutions and periodicity of a rational recursive
sequences of order five}, \textit{Bulletin of the Malaysian Mathematical Sciences
Society} \textbf{38:1} (2015), 95-112].


\noindent\ \ \ \ \ Mathematics Subject Classification: $39\mathrm{A}10,39%
\mathrm{A}99,39\mathrm{A}13$

\section{Introduction}

Recently, there has been a considerable interest in studying the dynamics of
rational difference equations. One of the reasons is that difference
equations have many applications in several mathematical models in biology,
ecology, physics, economics, genetics, population dynamics, medicine,
physiology and so forth, see for example \textbf{[\ref{bib.din1}], [\ref%
{bib.gocen1}], [\ref{bib.papaschinopoulos1}], [\ref{bib.zhou1}]}.
Furthermore, symmetry methods for differential equations especially Lie
symmetry analysis method have been investigated by several researchers. The
method of Lie symmetry has been applied to difference equations in the past
years and interesting results has been made by the authors (see \textbf{[\ref%
{bib.Haydon1}], [\ref{bib.Levi1}], [\ref{bib.quispel1}]}). The area of Lie
symmetry analysis is a very rich research field.

The Norwegian mathematician Sophus Lie studied the group of mappings which
leaves the differential equations invariant \textbf{[\ref{bib.Lie1}]}$.$
With this approach, one can solve difference equations using the group of
transformations that leaves the equations invariant, similarly. After the
inspiring works of Lie, the invariance properties of these equations
under groups of point transformations attracted great interest. The symmetry
method has been used to obtain the form of the solutions of difference
equations, see for example \textbf{[\ref{bib.Folly11}], [\ref{bib.Folly22}],
[\ref{bib.Folly33}], [\ref{bib.Nyirenda1}].} Moreover, Maeda \textbf{[\ref%
{bib.Maeda1}]} showed how to use symmetry methods and get the solutions of
the system of first-order ordinary difference equations.

In \textbf{[\ref{bib.Ibrahim1}]}, the author investigated the solutions and
properties of the difference equation 
\begin{equation*}
x_{n+1}=\frac{x_{n}x_{n-2}}{x_{n-1}(a+bx_{n}x_{n-2})}\text{ }\quad
\end{equation*}%
where initial values are nonnegative real numbers.

E.M. Elsayed and T. F. Ibrahim \textbf{[\ref{bib.Elsayed1}]} obtained the
solutions of the following difference equations of order five 
\begin{equation*}
x_{n+1}=\frac{x_{n}x_{n-2}x_{n-4}}{x_{n-1}x_{n-3}(\pm 1\pm
x_{n}x_{n-2}x_{n-4})},n=0,1,2,...
\end{equation*}
where the initial conditions $x_{-4},x_{-3},$ $x_{-2},x_{-1}$ and $x_{0}$
are arbitrary real numbers.

Our aim in this paper is to give the form of the solutions of the following
difference equation%
\begin{equation}  \label{xn}
x_{n+1}=\frac{x_{n}x_{n-2}x_{n-4}}{%
x_{n-1}x_{n-3}(a_{n}+b_{n}x_{n}x_{n-2}x_{n-4})}
\end{equation}%
in closed form, where $(a_{n})_{n\in \mathbb{N}_{0}},\ (b_{n})_{n\in \mathbb{%
N}_{0}}$ are non-zero real sequences, via the technique of Lie group
analysis. In this work, we are motivated by the results of E.M. Elsayed and T.
F. Ibrahim \textbf{[\ref{bib.Elsayed1}]}$.$

For the sake of convenience, we instead investigate the Kovalevskaya form of %
\eqref{xn}: 
\begin{equation}  \label{kov}
u_{n+5}=\frac{u_{n}u_{n+2}u_{n+4}}{%
u_{n+1}u_{n+3}(A_{n}+B_{n}u_{n}u_{n+2}u_{n+4})}.
\end{equation}

\section{ Preliminaries}

We begin by introducing some basic definitions and theorems needed in
the sequel. For details, see \textbf{[\ref{bib.Haydon1}].}

\begin{definition}
Let $G$ be a local group of transformations acting on a manifold $M$. $A$
subset $S\subset M$ is called $G$-invariant, and $G$ is called symmetry
group of $S$ , if whenever $x\in S$ , and $g\in G$ is such that $g.x$ is
defined, then $g.x\in S.$
\end{definition}

\begin{definition}
Let $G$ be a connected group of transformations acting on a manifold $M$. A
smooth real-valued function $V:M\rightarrow 
%TCIMACRO{\U{211d} }%
%BeginExpansion
\mathbb{R}
%EndExpansion
$ is an invariant function for $G$ if and only if 
\begin{equation}
X\left( V\right) =0\ \ \ \ for\text{ }all\ \ \ \ x\in M,
\end{equation}%
and every infinitesimal generator $X$ of $G$ .
\end{definition}

\begin{definition}
A parameterized set of point transformations, 
\begin{equation}
\Gamma _{\varepsilon }:x\rightarrow \hat{x}\left( x;\varepsilon \right) ,
\end{equation}%
where $x=x_{i}$, $i=1,...,p$ are continuous variables, is a one-parameter
local Lie group of transformations as long as the following conditions are
met:

\begin{description}
\item[a.] $\Gamma _{0}$ is the identity map if $\hat{x}=x$ when $\varepsilon
=0.$

\item[b.] $\Gamma _{a}\Gamma _{b}=\Gamma _{a+b}$ for every $a$ and $b$
sufficiently close to $0$.

\item[c.] Each $\hat{x}_{i}$ can be represented as a Taylor series (in a
neighborhood of $\varepsilon =0$ that is determined by $x$), and therefore%
\begin{equation*}
\hat{x}_{i}\left( x;\varepsilon \right) =x_{i}+\varepsilon \xi _{i}\left(
x\right) +O(\varepsilon ^{2}),i:1,...,p.
\end{equation*}
\end{description}
\end{definition}

Consider the $pth$-order difference equation%
\begin{equation}
u_{n+p}=\Omega \left( u_{n},...,u_{n+p-1}\right) ,  \label{bitsin}
\end{equation}%
for some smooth function $\Omega .$ Assume the point transformations are of
the form 
\begin{equation}
\hat{n}=n;\text{ \ \ \ \ \ \ \ \ \ \ \ \ \ \ }\hat{u}_{n}=u_{n}+\varepsilon
Q\left( n,u_{n}\right) O(\varepsilon ^{2}) \label{off}
\end{equation}%
with the corresponding infinitesimal symmetry generator%
\begin{equation*}
X=Q\left( n,u_{n}\right) \frac{\partial }{\partial u_{n}}+S^{(1)}Q\left(
n,u_{n}\right) \frac{\partial }{\partial u_{n+1}}+,,,+S^{(p-1)}Q\left(
n,u_{n}\right) \frac{\partial }{\partial u_{n+p-1}},
\end{equation*}%
where $S^{(j)}$ is the shift operator, i.e., $S^{(j)}:n\rightarrow n+j.$ The
symmetry condition is defined as 
\begin{equation}
\hat{u}_{n+p}=\Omega \left( n,\hat{u}_{n},\hat{u}_{n+1},...,\hat{u}%
_{n+p-1}\right) ,  \label{mirac}
\end{equation}%
whenever $\left( \ref{bitsin}\right) $ is true. Substituting the Lie point
symmetries $\left( \ref{off}\right) $ into the symmetry condition $\left( %
\ref{mirac}\right) $ leads to the linearized symmetry condition%
\begin{equation}
S^{(p)}Q-X\Omega =0|_{u_{n+p}=\Omega \left( u_{n},...,u_{n+p-1}\right).}  \label{yoruldum}
\end{equation}%
%whenever $\left( \ref{bitsin}\right) $ is holds. 
%Upon knowledge of the characteristic $Q$ , it is possible to obtain the invariant
%$V$ by introducing the canonical coordinate [\ref{bib.joshi1}]%
%\begin{equation*}
%S_{n}=\int \frac{du_{n}}{Q\left( n,u_{n}\right) }.
%\end{equation*}

\section{Reduction and solutions}

Consider difference equations of the form $\left( \ref{kov}\right) $, i.e.,%
\begin{equation}
u_{n+5}=\Omega =\frac{u_{n}u_{n+2}u_{n+4}}{%
u_{n+1}u_{n+3}(A_{n}+B_{n}u_{n}u_{n+2}u_{n+4})}.  \label{kalem}
\end{equation}%
To get the symmetry algebra of equation $\left( \ref{kalem}\right) ,$we
apply the invariance criterion $\left( \ref{yoruldum}\right) $ to $\left( %
\ref{kalem}\right) $ to get%
\begin{eqnarray}
&&Q(n+5,\Omega )-\frac{u_{n+2}u_{n+4}A_{n}}{%
u_{n+1}u_{n+3}(A_{n}+B_{n}u_{n}u_{n+2}u_{n+4})^{2}}Q(n,u_{n})  \notag \\
&&+\frac{u_{n+2}u_{n+4}u_{n}}{%
u_{n+3}u_{n+1}^{2}(A_{n}+B_{n}u_{n}u_{n+2}u_{n+4})}Q(n+1,u_{n+1})  \notag \\
&&-\frac{u_{n+4}u_{n}A_{n}}{%
u_{n+1}u_{n+3}(A_{n}+B_{n}u_{n}u_{n+2}u_{n+4})^{2}}Q(n+2,u_{n+2})  \notag \\
&&+\frac{u_{n+4}u_{n+2}u_{n}}{%
u_{n+1}u_{n+3}^{2}(A_{n}+B_{n}u_{n}u_{n+2}u_{n+4})}Q(n+3,u_{n+3})  \notag \\
&&-\frac{u_{n+2}u_{n}A_{n}}{%
u_{n+1}u_{n+3}(A_{n}+B_{n}u_{n}u_{n+2}u_{n+4})^{2}}Q(n+4,u_{n+4})=0.
\label{overwhelm}
\end{eqnarray}%
By acting the partial differential operator 
\begin{equation*}
L=\frac{\partial }{\partial u_{n}}-\frac{\partial u_{n+3}}{\partial u_{n}}%
\frac{\partial }{\partial u_{n+3}}=\frac{\partial }{\partial u_{n}}-\frac{%
(\partial \Omega /\partial u_{n})}{(\partial \Omega /\partial u_{n+3})}\frac{%
\partial }{\partial u_{n+3}},
\end{equation*}%
on \eqref{overwhelm}, we have%
\begin{align}
&\left(2u_{n+2}u_{n}B_{n}+\frac{A_{n}}{u_{n}}%
\right)Q-(A_{n}+B_{n}u_{n}u_{n+2}u_{n+4})Q^{^{\prime }}
+(B_{n}u_{n}u_{n+4})S^{(2)}Q  \notag \\
&-\frac{\left( A_{n}+B_{n}u_{n}u_{n+2}u_{n+4}\right) }{u_{n+3}}%
S^{(3)}Q^{^{\prime }}+\left( B_{n}u_{n}u_{n+2}\right) S^{(4)}Q=0.
\label{kiraz}
\end{align}%
The notation $\prime $ stands for the derivative with respect to the
continuous variable. Differentiation of $\left( \ref{kiraz}\right) $ with
respect to $u_{n}$ twice, while $u_{n+2},$ $u_{n+3}$ and $u_{n+4}$ are
fixed, yields the equation%
\begin{equation}
\frac{2A_{n}}{u_{n}^{3}}Q-\frac{2A_{n}}{u_{n}^{2}}Q^{^{\prime }}+\frac{A_{n}%
}{u_{n}}Q^{\prime \prime }-A_{n}Q^{^{\prime \prime \prime
}}-(B_{n}u_{n}u_{n+2}u_{n+4})Q^{^{\prime \prime \prime }}=0.  \label{hacer}
\end{equation}%
The equation above is solved by separation of variables in powers of shifts
of $u_{n}.$ Thus, we have the system 
\begin{eqnarray*}
1 &:&Q^{^{\prime \prime \prime }}(n,u_{n})-\frac{1}{u_{n}}Q^{\prime \prime
}(n,u_{n})+\frac{2}{u_{n}^{2}}Q^{^{\prime }}(n,u_{n})-\frac{2}{u_{n}^{3}}%
Q(n,u_{n})=0 \\
u_{n+2}u_{n+4} &:&Q^{^{\prime \prime \prime }}(n,u_{n})=0
\end{eqnarray*}%
whose solution is as follows:%
\begin{equation}
Q(n,u_{n})=\alpha _{n}u_{n}^{2}+\beta _{n}u_{n},  \label{yagmur}
\end{equation}%
for some arbitrary functions $\alpha _{n}$ and $\beta _{n}$ of $n$ to be
found. Substituting $\left( \ref{yagmur}\right) $ and its first, second and
third shifts in $\left( \ref{overwhelm}\right) $ , and thereafter,
replacing the expression of $u_{n+5}$ given in $\left( \ref{kalem}\right) $
in the obtained equation, it follows that 
\begin{align}
&u_{n}^{2}u_{n+2}^{2}u_{n+4}^{2}\alpha _{n+5}-A_{n}\alpha
_{n}u_{n}^{2}u_{n+1}u_{n+2}u_{n+3}u_{n+4}+A_{n}\alpha
_{n+1}u_{n}u_{n+1}^{2}u_{n+2}u_{n+3}u_{n+4}  \notag \\
&-A_{n}\alpha _{n+2}u_{n}u_{n+1}u_{n+2}^{2}u_{n+3}u_{n+4}+A_{n}\alpha
_{n+3}u_{n}u_{n+1}u_{n+2}u_{n+3}^{2}u_{n+4}  \notag \\
&-A_{n}\alpha _{n+4}u_{n}u_{n+1}u_{n+2}u_{n+3}u_{n+4}^{2} +\alpha
_{n+1}B_{n}u_{n}^{2}u_{n+2}^{2}u_{n+4}^{2}u_{n+3}u_{n+1}^{2}  \notag \\
&+\alpha
_{n+3}B_{n}u_{n}^{2}u_{n+2}^{2}u_{n+4}^{2}u_{n+3}^{2}u_{n+1}+A_{n}u_{n}u_{n+1}u_{n+2}u_{n+3}u_{n+4}(\beta _{n+5}-\beta _{n}+\beta _{n+1}
\notag \\
&-\beta _{n+2}+\beta _{n+3}-\beta _{n+4})
+B_{n}u_{n}^{2}u_{n+2}^{2}u_{n+4}^{2}u_{n+1}u_{n+3}(\beta _{n+5}+\beta
_{n+1}+\beta _{n+3})=0.
\end{align}%
Equating coefficients of all powers of shifts of $u_{n}$ to zero and
simplifying the resulting system , we get its reduced form%
\begin{equation}
\alpha _{n}=0,\text{ \ \ \ \ \ \ \ \ }\beta _{n}+\beta _{n+2}+\beta _{n+4}=0
\label{cay}
\end{equation}%
and its solutions are%
\begin{equation*}
\alpha _{n}=0\text{\ \ \ \ \ and \ \ \ \ }\beta _{n}=(-1)^{n}\beta
^{n}c_{1}+(-1)^{n}\overline{\beta }^{n}c_{2}+\overline{\beta }%
^{n}c_{3}+\beta ^{n}c_{4}
\end{equation*}%
for some arbitrary constants $c_{i},i=1,...,4$, and where $\beta =\exp
\left( \pi i/3\right) .$ So, we get four characteristics given by%
\begin{align}
&Q_{1}(n,u_{n})=(-1)^{n}\beta ^{n}u_{n},Q_{2}(n,u_{n})=(-1)^{n}\overline{%
\beta }^{n}u_{n},Q_{3}(n,u_{n})=\overline{\beta }^{n}u_{n},  \notag \\
&Q_{4}(n,u_{n})=\beta ^{n}u_{n}.
\end{align}%
The four corresponding symmetry generators $X_{1},X_{2},X_{3}$ and $X_{4}$
are as follows:%
\begin{align}
& X_{1}=\sum_{j=0}^{4}(-\beta)^{n+j}u_{n+j}\partial u_{n+j},\quad
X_{2}=\sum_{j=0}^{4}(-\bar{\beta})^{n+j}u_{n+j}\partial u_{n+j},  \notag \\
&X_{3}=\sum_{j=0}^{4}(\bar{\beta})^{n+j}u_{n+j}\partial u_{n+j},\quad
X_{4}=\sum_{j=0}^{4}(\beta)^{n+j}u_{n+j}\partial u_{n+j}.
\end{align}%
Here, using $Q_{4}$, we introduce the canonical coordinate [\ref{bib.joshi1}]%
\begin{equation}
S_{n}=\int \frac{du_{n}}{Q_{4}(n,u_{n})}=\int \frac{du_{n}}{\beta ^{n}u_{n}}=%
\frac{1}{\beta ^{n}}\ln \left \vert u_{n}\right \vert .  \label{star}
\end{equation}%
Using relation $\left( \ref{cay}\right) $, we derive the function 
\begin{subequations}
\label{serapsevim}
\begin{equation}
\widetilde{V}_{n}=S_{n}\beta ^{n}+S_{n+2}\beta ^{n+2}+S_{n+4}\beta ^{n+4}
\label{serap}
\end{equation}%
%
%
%
%
%It is easy to check that
%\begin{equation*}
%X_{i}(\widetilde{V}_{n})=\beta ^{n+4}+\beta ^{n+2}+\beta ^{n}=0\text{ \ \ \
%\ }i=1,2,3,4,
%\end{equation*}%
%For the sake of convenience, we use%
and let 
\begin{equation}
\left \vert V_{n}\right \vert =\exp \left \{ -\widetilde{V}_{n}\right \} .
\label{sevim}
\end{equation}%
In other words, $V_{n}=\pm 1/(u_{n}u_{n+2}u_{n+4}).$ One can show by using $%
\left( \ref{kalem}\right) $ and $\left( \ref{sevim}\right) $ that 
\end{subequations}
\begin{eqnarray}
V_{n+1} &=&A_{n}V_{n}+B_{n},
\end{eqnarray}
that is, 
\begin{eqnarray}
V_{n} &=&V_{0}\underset{k_{1}=0}{\overset{n-1}{\prod }}A_{k_{1}}+\overset{n-1%
}{\underset{l=0}{\sum }}B_{l}\underset{l+1=0}{\overset{n-1}{\prod }}%
A_{k_{2}}.  \label{buneya}
\end{eqnarray}%
Here, we first use $\left( \ref{star}\right) $ to have%
\begin{equation*}
\left \vert u_{n}\right \vert =\exp (S_{n}\beta ^{n}).
\end{equation*}%
Then, invoking $\left( \ref{serapsevim}\right) $ yields%
\begin{eqnarray}
\left \vert u_{n}\right \vert&=&|H_{n}|\exp \left[ \frac{2\sqrt{3}}{3}%
\overset{n-1}{\underset{k=0}{\sum }}\left( \cos \frac{(n-k)\pi }{2}\cos 
\frac{(n-k+1)\pi }{6}\right) \ln \left \vert V_{k}\right \vert \right] ,
\label{secil}
\end{eqnarray}%
where $V_{k}$ is given in $\left( \ref{buneya}\right) $ 
%and  $H_{n}=\exp \left((-\beta )^{n}c_{5}+(\beta )^{n}c_{6}+(\bar{\beta})^{n}c_{7}+(-\bar{\beta}%)^{n}c_{8}\right) $ , $\gamma \left( n,k\right) =\cos \frac{(n-k)\pi }{2}%\cos \frac{(n-k+1)\pi }{6}$
and note that the $H_n$'s satisfy%
\begin{equation*}
H_{0}=x_{0},\text{ \ }H_{1}=x_{1},\text{ \ }H_{2}=x_{2},\text{ \ }%
H_{3}=x_{3},\text{ \ }H_{4}=\frac{1}{x_{0}x_{2}},\text{ \ }H_{5}=\frac{1}{%
x_{1}x_{3}},\text{ \ }H_{6n+j}=H_{j}.
\end{equation*}%
We can simplify the solution $\left( \ref{secil}\right) $ further by
splitting it into six categories. We have
\begin{eqnarray}
\left \vert u_{6n}\right \vert &=&H_{6n}\exp \left[ \frac{2\sqrt{3}}{3}%
\overset{6n-1}{\underset{k=0}{\sum }}\left( \cos \frac{(6n-k)\pi }{2}\cos 
\frac{(6n-k+1)\pi }{6}\right) \ln \left \vert V_{k}\right \vert \right] 
\notag \\
&=&H_{0}\exp \left( \ln V_{0}-\ln V_{2}+\ln V_{6}-\ln V_{8}+...+\ln
V_{6n-6}-\ln V_{6n-4}\right)  \notag \\
&=& x_{0} \overset{n-1}{\underset{k=0}{\prod }}\left \vert \frac{V_{6k}}{%
V_{6k+2}}\right \vert.  \label{gece}
\end{eqnarray}%
%
%
%
%
%To find $\exp \left( H_{0}\right) $ , set $n=0$ in $\left( \ref{gece}\right)
%$ and note that $\left \vert u_{0}\right \vert =\exp \left( H_{0}\right) .$
Thus%
\begin{equation}
u_{6n}=u_{0}\underset{k=0}{\overset{n-1}{\prod }}\frac{V_{6k}}{V_{6k+2}}.
\label{ay}
\end{equation}%
It can be verified, using $V_{n}=1/(u_{n}u_{n+2}u_{n+4}),$ that there is
no need \ for absolute value function in $\left( \ref{ay}\right) $.
Similarly, for any $j=0,\dots, 5,$ we obtain the following:%
\begin{equation}
u_{6n+j}=u_{j}\underset{k=0}{\overset{n-1}{\prod }}\frac{V_{6k+j}}{V_{6k+j+2}%
}.  \label{hahahaa}
\end{equation}%
For $j=0,$ we get%
\begin{eqnarray*}
u_{6n} &=&u_{0}\underset{k=0}{\overset{n-1}{\prod }}\frac{V_{6k}}{V_{6k+2}}
\\
&=&u_{0}\underset{k=0}{\overset{n-1}{\prod }}\frac{\left(\underset{k_{1}=0}{%
\overset{6k-1}{\prod }}A_{k_{1}}\right)+u_{0}u_{2}u_{4}\overset{6k-1}{\underset{l=0}%
{\sum }}\left(B_{l}\underset{k_{2}=l+1}{\overset{6k-1}{\prod }}A_{k_{2}}\right)}{\left(\underset%
{k_{1}=0}{\overset{6k+1}{\prod }}A_{k_{1}}\right)+u_{0}u_{2}u_{4}\overset{6k+1}{%
\underset{l=0}{\sum }}\left(B_{l}\underset{k_{2}=l+1}{\overset{6k+1}{\prod }}%
A_{k_{2}}\right)}.
\end{eqnarray*}%
For $j=1,$ we have%
\begin{eqnarray*}
u_{6n+1} &=&u_{1}\underset{k=0}{\overset{n-1}{\prod }}\frac{V_{6k+1}}{%
V_{6k+3}} \\
&=&u_{1}\underset{k=0}{\overset{n-1}{\prod }}\frac{\left(\underset{k_{1}=0}{%
\overset{6k}{\prod }}A_{k_{1}}\right)+u_{0}u_{2}u_{4}\overset{6k}{\underset{l=0}{%
\sum }}\left(B_{l}\underset{k_{2}=l+1}{\overset{6k}{\prod }}A_{k_{2}}\right)}{\left(\underset{%
k_{1}=0}{\overset{6k+2}{\prod }}A_{k_{1}}\right)+u_{0}u_{2}u_{4}\overset{6k+2}{%
\underset{l=0}{\sum }}\left(B_{l}\underset{k_{2}=l+1}{\overset{6k+2}{\prod }}%
A_{k_{2}}\right)}.
\end{eqnarray*}%
For $j=2,$ we have%
\begin{eqnarray*}
u_{6n+2} &=&u_{2}\underset{k=0}{\overset{n-1}{\prod }}\frac{V_{6k+2}}{%
V_{6k+4}} \\
&=&u_{2}\underset{k=0}{\overset{n-1}{\prod }}\frac{\left(\underset{k_{1}=0}{%
\overset{6k+1}{\prod }}A_{k_{1}}\right)+u_{0}u_{2}u_{4}\overset{6k+1}{\underset{l=0}%
{\sum }}\left(B_{l}\underset{k_{2}=l+1}{\overset{6k+1}{\prod }}A_{k_{2}}\right)}{\left(\underset%
{k_{1}=0}{\overset{6k+3}{\prod }}A_{k_{1}}\right)+u_{0}u_{2}u_{4}\overset{6k+3}{%
\underset{l=0}{\sum }}\left(B_{l}\underset{k_{2}=l+1}{\overset{6k+3}{\prod }}%
A_{k_{2}}\right)}.
\end{eqnarray*}%
For $j=3,$ we have%
\begin{eqnarray*}
u_{6n+3} &=&u_{3}\underset{k=0}{\overset{n-1}{\prod }}\frac{V_{6k+3}}{%
V_{6k+5}} \\
&=&u_{3}\underset{k=0}{\overset{n-1}{\prod }}\frac{\left(\underset{k_{1}=0}{%
\overset{6k+2}{\prod }}A_{k_{1}}\right)+u_{0}u_{2}u_{4}\overset{6k+2}{\underset{l=0}%
{\sum }}\left(B_{l}\underset{k_{2}=l+1}{\overset{6k+2}{\prod }}A_{k_{2}}\right)}{\left(\underset%
{k_{1}=0}{\overset{6k+4}{\prod }}A_{k_{1}}\right)+u_{0}u_{2}u_{4}\overset{6k+4}{%
\underset{l=0}{\sum }}\left(B_{l}\underset{k_{2}=l+1}{\overset{6k+4}{\prod }}%
A_{k_{2}}\right)}.
\end{eqnarray*}%
For $j=4,$ we have%
\begin{eqnarray*}
u_{6n+4} &=&u_{4}\underset{k=0}{\overset{n-1}{\prod }}\frac{V_{6k+4}}{%
V_{6k+6}} \\
&=&u_{4}\underset{k=0}{\overset{n-1}{\prod }}\frac{\left(\underset{k_{1}=0}{%
\overset{6k+3}{\prod }}A_{k_{1}}\right)+u_{0}u_{2}u_{4}\overset{6k+3}{\underset{l=0}%
{\sum }}\left(B_{l}\underset{k_{2}=l+1}{\overset{6k+3}{\prod }}A_{k_{2}}\right)}{\left(\underset%
{k_{1}=0}{\overset{6k+5}{\prod }}A_{k_{1}}\right)+u_{0}u_{2}u_{4}\overset{6k+5}{%
\underset{l=0}{\sum }}\left(B_{l}\underset{k_{2}=l+1}{\overset{6k+5}{\prod }}%
A_{k_{2}}\right)}.
\end{eqnarray*}%
Finally, for $j=5,$ we have%
\begin{eqnarray*}
u_{6n+5} &=&u_{5}\underset{k=0}{\overset{n-1}{\prod }}\frac{V_{6k+5}}{%
V_{6k+7}} \\
&=&u_{5}\underset{k=0}{\overset{n-1}{\prod }}\frac{\left(\underset{k_{1}=0}{%
\overset{6k+4}{\prod }}A_{k_{1}}\right)+u_{0}u_{2}u_{4}\overset{6k+4}{\underset{l=0}%
{\sum }}\left(B_{l}\underset{k_{2}=l+1}{\overset{6k+4}{\prod }}A_{k_{2}}\right)}{\left(\underset%
{k_{1}=0}{\overset{6k+6}{\prod }}A_{k_{1}}\right)+u_{0}u_{2}u_{4}\overset{6k+6}{%
\underset{l=0}{\sum }}\left(B_{l}\underset{k_{2}=l+1}{\overset{6k+6}{\prod }}%
A_{k_{2}}\right)}.
\end{eqnarray*}%
More compactly, 
% using $\left( \ref{buneya}\right) $ and $\left( \ref{hahahaa}%
%\right) $
we have the solutions of $\left( \ref{kalem}\right) $ as follows:%
\begin{equation*}
u_{6n+j}=u_{j}\underset{k=0}{\overset{n-1}{\prod }}\frac{\left(\underset{k_{1}=0}{%
\overset{6k+j-1}{\prod }}A_{k_{1}}\right)+u_{0}u_{2}u_{4}\overset{6k+j-1}{\underset{%
l=0}{\sum }}\left(B_{l}\underset{k_{2}=l+1}{\overset{6k+j-1}{\prod }}A_{k_{2}}\right)}{%
\left(\underset{k_{1}=0}{\overset{6k+j+1}{\prod }}A_{k_{1}}\right)+u_{0}u_{2}u_{4}\overset%
{6k+j+1}{\underset{l=0}{\sum }}\left(B_{l}\underset{k_{2}=l+1}{\overset{6k+j+1}{%
\prod }}A_{k_{2}}\right)},
\end{equation*}%
$j=0, 1, 2, 3, 4, 5$. Therefore the solutions of $\left( \ref{xn}%
\right) $ become%
\begin{equation}  \label{solxn}
x_{6n+j-4}=x_{j-4}\underset{k=0}{\overset{n-1}{\prod }}\frac{\left(\underset{%
k_{1}=0}{\overset{6k+j-1}{\prod }}a_{k_{1}}\right)+x_{-4}x_{-2}x_{0}\overset{6k+j-1}%
{\underset{l=0}{\sum }}\left(b_{l}\underset{k_{2}=l+1}{\overset{6k+j-1}{\prod }}%
a_{k_{2}}\right)}{\left(\underset{k_{1}=0}{\overset{6k+j+1}{\prod }}%
a_{k_{1}}\right)+x_{-4}x_{-2}x_{0}\overset{6k+j+1}{\underset{l=0}{\sum }}\left(b_{l}%
\underset{k_{2}=l+1}{\overset{6k+j+1}{\prod }}a_{k_{2}}\right)},0\leq j\leq 5.
\end{equation}

\subsection{The case $a_{n}$ and $b_{n}$ are 1-periodic sequences.}

In this case, $a_{n}=a$ and $b_{n}=b$ where $a,\ b\in \mathbb{R}$. Then,
equations in \eqref{solxn} simplify to 
\begin{equation}  \label{solxncst}
x_{6n+j-4}=x_{j-4}\underset{k=0}{\overset{n-1}{\prod }}\frac{%
a^{6k+j}+x_{-4}x_{-2}x_{0}b\left(\overset{6k+j-1}{\underset{l=0}{\sum }} a^l\right)}{%
a^{6k+j+2}+x_{-4}x_{-2}x_{0}b\left(\overset{6k+j+1}{\underset{l=0}{\sum }}a^l\right)}%
,0\leq j\leq 5.
\end{equation}

\begin{theorem}
The difference equation 
\begin{equation}
x_{n+1}=\frac{x_{n}x_{n-2}x_{n-4}}{x_{n-1}x_{n-3}(a+bx_{n}x_{n-2}x_{n-4})}
\label{six}
\end{equation}%
has a periodic solution of period six if and only if $x_{-4}x_{-2}x_{0}=%
\frac{1-a}{b}$, $a\neq 1$.
\end{theorem}

\begin{proof}
Here, $x_{-4}x_{-2}x_{0}=\frac{1-a}{b}$ and then \eqref{solxncst} simplifies
to 
\begin{align*}
x_{6n+j-4}=& x_{j-4}\underset{k=0}{\overset{n-1}{\prod }}\frac{a^{6k+j}+(1-a)%
\left(\overset{6k+j-1}{\underset{l=0}{\sum }}a^{l}\right)}{a^{6k+j+2}+(1-a)\left(\overset{6k+j+1%
}{\underset{l=0}{\sum }}a^{l}\right)},0\leq j\leq 5 \\
=& x_{j-4}\underset{k=0}{\overset{n-1}{\prod }}\frac{a^{6k+j}+(1-a)\left(\frac{%
1-a^{6k+j}}{1-a}\right)}{a^{6k+j+2}+(1-a)\left(\frac{1-a^{6k+j+2}}{1-a}\right)},0\leq j\leq 5 \\
=& x_{j-4},
\end{align*}%
for $i=0,\dots 5$.
%\newpage 
\end{proof}
%\newpage
\begin{example}
Consider the Eq. (\ref{six}) where $a=-1,$ $b=1$ and the initial conditions $%
x_{-4}=0.2,$ $x_{-3}=9,$ $x_{-2}=5,$ $x_{-1}=7,$\ $x_{0}=2$ to verify our
theoretical results. 
\begin{figure}[h]
\centering
\includegraphics[width=.6\linewidth]{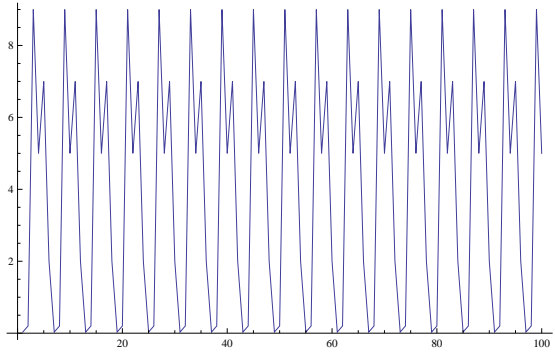}
\caption{Plot of $x_{n+1}=\frac{x_{n}x_{n-2}x_{n-4}}{%
x_{n-1}x_{n-3}(-1+x_{n}x_{n-2}x_{n-4})}$}
\end{figure}
\end{example}
\begin{theorem}
Let $\{x_n\}$ be a well-defined solution to the sequence 		
	\begin{equation}
	x_{n+1}=\frac{x_{n}x_{n-2}x_{n-4}}{%
		x_{n-1}x_{n-3}(1+bx_{n}x_{n-2}x_{n-4})}
	\end{equation}%
	with $b\neq 0$. Then 
	\begin{equation*}
	\lim _{n\rightarrow \infty} x_n =0.
	\end{equation*}%
\end{theorem}
\begin{proof}
Using \eqref{solxncst}, with $a=1$, we have 
	\begin{align*}
	x_{6n+j-4}=& x_{j-4}\underset{k=0}{\overset{n-1}{\prod }}\frac{1+(6k+j)bx_{-4}x_{-2}x_{0}}{1+(6k+j+2)bx_{-4}x_{-2}x_{0}} \\
	=& x_{j-4}\underset{k=0}{\overset{n-1}{\prod }}\left[1-\frac{2bx_{-4}x_{-2}x_{0}}{1+(6k+j+2)bx_{-4}x_{-2}x_{0}} \right].
	\end{align*}%
	Clearly, $1+(6k+j+2)bx_{-4}x_{-2}x_{0} \rightarrow \infty$ as $k \rightarrow \infty$. Therefore, we can always find a sufficiently large $n_0\in \mathbb{N}$ such that for all $k>n_0$, we have 
	\begin{align*}
	1+(6k+j+2)bx_{-4}x_{-2}x_{0} \simeq (6k+j+2)bx_{-4}x_{-2}x_{0}.
	\end{align*}
	So
	\begin{align*}
	x_{6n+j-4}=& x_{j-4}\Gamma(n_0)\underset{k=n_0+1}{\overset{n-1}{\prod }}\left(1-\frac{2}{6k+j+2} \right) \\
	=& x_{j-4}\Gamma(n_0)\underset{k=n_0+1}{\overset{n-1}{\prod }}\exp \left[\ln\left(1-\frac{2}{6k+j+2} \right)\right]. 
	\end{align*}
	Note that 
\begin{align}
\Gamma(n_0)=\underset{k=0}{\overset{n_0}{\prod }}\left(1-\frac{2}{6k+j+2} \right).
\end{align}
Now, using the fact that $\ln(1+x)=x+O(x^2)$ for small $x$, we obtain
\begin{align*}
x_{6n+j-4}	=& x_{j-4}\Gamma(n_0)\underset{k=n_0+1}{\overset{n-1}{\prod }}\exp \left[-\frac{2}{6k+j+2} +O\left(\frac{1}{(6k+j+2)^2}\right)\right] \\
	=&  x_{j-4}\Gamma(n_0)\exp \left[-\underset{k=n_0+1}{\overset{n-1}{\sum }}\left(\frac{2}{6k+j+2}\right) \right]    
	\underset{k=n_0+1}{\overset{n-1}{\prod }}\exp \left[O\left(\frac{1}{(6k+j+2)^2}\right)\right],
	\end{align*}
$j=0,\dots,5$,	and hence $x_n \rightarrow 0$ as $n \rightarrow \infty$.
%	\begin{equation*}
%	\lim _{n\rightarrow \infty} x_n =0.
%	\end{equation*}
\end{proof}
\\ \\
\par \noindent 
Next, we deal with the cases $a=\pm 1$ and $b=\pm 1.$
\subsubsection{\noindent Case : $a=1$ and $b=\pm 1.$}
\bigskip\ We can verify our results from [\ref{bib.Elsayed1}] (see Theorems
1 and 6) that
\begin{equation*}
x_{6n+j-4}=x_{j-4}\underset{k=0}{\overset{n-1}{\prod }}\frac{%
1+b(6k+j)x_{-4}x_{-2}x_{0}}{1+b(6k+j+2)x_{-4}x_{-2}x_{0}}.
\end{equation*}%
Clearly, in terms of the initial values $x_{-4},x_{-3},x_{-2},\ x_{-1},\
x_{0}$, we have 
\begin{equation*}
x_{6n-4}=x_{-4}\prod_{k=0}^{n-1}\frac{1+6bkx_{-4}x_{-2}x_0}{%
1+(6k+2)bx_{-4}x_{-2}x_{0}},
\end{equation*}%
\begin{equation*}
x_{6n-3}=x_{-3}\prod_{k=0}^{n-1}\frac{1+(6k+1)bx_{-4}x_{-2}x_0}{%
1+(6k+3)bx_{-4}x_{-2}x_{0}},
\end{equation*}%
\begin{equation*}
x_{6n-2}=x_{-2}\prod_{k=0}^{n-1}\frac{1+(6k+2)bx_{-4}x_{-2}x_0}{%
1+(6k+4)bx_{-4}x_{-2}x_{0}},
\end{equation*}%
\begin{equation*}
x_{6n-1}=x_{-2}\prod_{k=0}^{n-1}\frac{1+(6k+3)bx_{-4}x_{-2}x_0}{%
1+(6k+5)bx_{-4}x_{-2}x_{0}},
\end{equation*}%
\begin{equation*}
x_{6n}=x_{0}\prod_{k=0}^{n-1}\frac{1+(6k+4)bx_{-4}x_{-2}x_0}{%
1+(6k+6)bx_{-4}x_{-2}x_{0}}
\end{equation*}
and 
\begin{equation*}
x_{6n+1}=\frac{x_{-4}x_{-2}x_0}{x_{-1}x_{-3}(1+bx_{-4}x_{-2}x_0)}%
\prod_{k=0}^{n-1}\frac{1+(6k+5)bx_{-4}x_{-2}x_0}{1+(6k+7)bx_{-4}x_{-2}x_{0}%
}.
\end{equation*}

\begin{example}
Consider the Eq. (\ref{six}) where $a,$ $b=1$ and the initial conditions $%
x_{-4}=-0.2,$ $x_{-3}=3,$ $x_{-2}=1.3,$ $x_{-1}=0.7,$\ $x_{0}=-2$ to verify
our theoretical results. 
\begin{figure}[h]
\centering
\includegraphics[width=.6\linewidth]{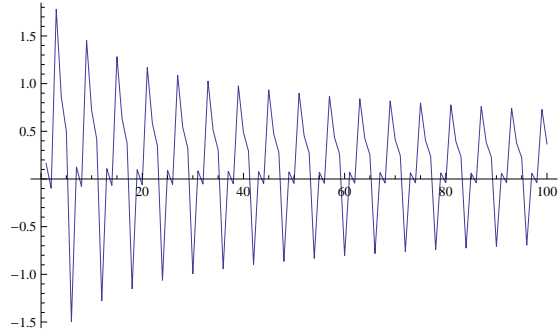}
\caption{Plot of $x_{n+1}=\frac{x_{n}x_{n-2}x_{n-4}}{%
x_{n-1}x_{n-3}(1+x_{n}x_{n-2}x_{n-4})}$}
\end{figure}
\end{example}

\subsubsection{Case : $a=-1$ and $b=\pm 1.$}

In this case, the solution becomes%
\begin{eqnarray*}
x_{6n+j-4} &=&x_{j-4}\prod_{k=0}^{n-1}\frac{(-1)^{j}+bx_{-4}x_{-2}x_{0}\left(%
\dfrac{1-(-1)^{j}}{2}\right)}{(-1)^{j}+bx_{-4}x_{-2}x_{0}\left(\dfrac{1-(-1)^{j}}{2}\right)}
\\
&=&x_{j-4}
\end{eqnarray*}%
where $j=0,1,2,3,4,5$.

\noindent Therefore, in this case it is easy to see that the solutions are
periodic with period six. It is clear that, here, the extra condition $%
x_{-4}x_{-2}x_0=\frac{1-a}{b}$ in the above theorem is not needed. We can
check our results from [\ref{bib.Elsayed1}] (see Theorems 3 and 8).

\section{Conclusion}

Lie symmetry generators of difference equations of the form \eqref{kov} 
%\begin{equation*}
%u_{n+5}=\frac{u_{n}u_{n+2}u_{n+4}}{%
%u_{n+1}u_{n+3}(A_{n}+B_{n}u_{n}u_{n+2}u_{n+4})}
%\end{equation*}%
were obtained, which in turn were useful in finding the solutions to the
difference equation  \eqref{xn}. %
%\begin{equation*}
%x_{n+1}=\frac{x_{n}x_{n-2}x_{n-4}}{%
%x_{n-1}x_{n-3}(a_{n}+b_{n}x_{n}x_{n-2}x_{n-4})}.
%\end{equation*}%
Closed form formulas for the solutions were given and specific cases
exhibited. Furthermore, periodic nature and behavior of solutions  for some special cases were discussed.

\end{document}